\documentclass[a4paper,10pt]{article}

\usepackage{amsmath,amsthm,amssymb,sansmath,color,geometry,multicol,graphicx, float,color,authblk}
\usepackage[english]{babel}
\usepackage[utf8]{inputenc}
\newtheorem{rmq}{Remark}
\newtheorem{lem}{Lemma}
\newtheorem{thm}{Theorem}

\newtheorem{hyp}{Assumption}

\newcommand{\x}{\partial_x}
\newcommand{\y}{\partial_y}
\newcommand{\po}{\left(}
\newcommand{\pf}{\right)}

\title{Hypocoercive relaxation to equilibrium for some kinetic models\\
\emph{via} a third order differential inequality}
\author{Pierre Monmarché}
\affil{Institut de Mathématiques de Toulouse}

\makeindex
\begin{document}

\maketitle
\abstract{This paper deals with the study of some particular kinetic models, where the randomness acts only on the velocity variable level. Usually, the Markovian generator cannot satisfy any Poincaré's inequality. Hence, no Gronwall's lemma can easily lead to the exponential decay of $F_t$ (the $L^2$ norm of a test function along the semi-group). Nevertheless for the kinetic Fokker-Planck dynamics and for a piecewise deterministic evolution we show that $F_t$ satisfies a third order differential inequality which gives an explicit rate of convergence to equilibrium.}

\tableofcontents

\section{Introduction}

In order to improve MCMC algorithms one can try to resort to higher order dynamics, for instance kinetic ones.
Indeed, non-reversible dynamics naturally possess more inertia than reversible ones and have less tendency to turn back and hesitate than the simple reversible process. This is an important issue for the escape of local minima and such non-reversible processes may then converge faster to equilibrium (cf. \cite{Diaconis2000}, \cite{DiaconisMiclo}, \cite{Neal2004}, \cite{Lelievre2006}). 

For instance, \cite{Lelievre2006} compare numerically the following sampling procedures of the Gibbs measure $e^{-U(x)}dx$ associated to a potential $U$.
 First, thanks to the Fokker-Planck dynamics
\[dX_t = -U'(X_t) + \sigma dB_t\]
and secondly with the kinetic Fokker-Planck one (shorten from now on to kFP ; it is called Langevin dynamics in \cite{Lelievre2006}, but we stick here to \cite{DesVilla2003} for the denomination)

\begin{equation}\label{equationkFP}
\begin{cases}
& dX_t  =  Y_tdt\\
& dY_t  =  -U'(X_t)dt - Y_tdt + \sqrt2 dB_t
\end{cases}
\end{equation}
where $B_t$ stands for a standard brownian motion. It turns out, numerically, that the second one is generally more efficient, in the sense that it converges faster toward the steady regime.

\bigskip

About a decade ago, there were no method to obtain explicit rates of convergence for non-reversible Markov process, as usualy the classical functional inequality theory (cf. \cite{logSob}, \cite{Markowich99}), powerfull in reversible settings, does not apply (it can in some particular cases, see \cite{Arnold1}, \cite{Arnold2}). But since then, as the topic is of interest in many fields, many different approaches have emerged. Here is a far from exhaustive list of references roughly sorted in three group : first the analytical method based on the spectral study of hypoelliptical operators, initiated by Hérau and Nier (in \cite{Nier2004}, followed by \cite{Herau2005}, \cite{Hairer2003}, \cite{Helffer2005}, \cite{Zegarlinski}), where the decay is obtained in some Sobolev norm. Secondly the probabilistic method of coupling \emph{à la} Meyn and Tweedie, in Wasserstein distances (see \cite{GadatPanloup2013},\cite{Guillin2006},\cite{Guillin2011}, \cite{Bolley2010}, and \cite{Cattiaux2008} for a link with functional 
inequalities), recently succesfully applied in particular in the 
field of PDMPs (piecewise 
deterministic Markovian processes ; see \cite{Malrieu2011}, \cite{Fontbona2010}, \cite{ChafaiMalrieuParoux}, \cite{BLBMZ2}). Finally the method of the modified Lyapunov function initiated by Desvillettes and Villani (\cite{Desvillettes2003}, \cite{DesVilla2003}, \cite{Desvillettes2006}, \cite{Villani2009},\cite{Villani2}, \cite{MouhotNeumann}, \cite{Fellner2003}, \cite{Calogero2010}), and then refined by Dolbeault, Mouhot and Schmeiser (\cite{DMS2009}, \cite{DMS2011}, \cite{Grothaus1}, \cite{Grothaus2}) who work for the latter with a norm equivalent to the $L^2$ one, without any addition of supplementary derivatives. The present work is rather close to this last approcah.

\bigskip

Despite (or thanks to) all this work, some phenomena arising from the interplay between the deterministic transport and the stochastic part of the generator still deserve to be better understood. In particular the convergence to equilibrium appears to be inhomogeneous in time: in \cite{Gadat2013}, where the $L^2$ distance $d(t)$ between the distribution at time $t$ and the equilibrium is explicitly computed for the kFP process with a quadratic potential, the decay is flat for small times, \emph{i.e.} $d(t) \simeq 1 - c t^3$. Indeed, if $d'(0)$  were non zero, it would imply a Poincaré inequality (see \cite{logSob}) but none is satisfied there. Furthermore in some cases we have $d(t)=g_t e^{-\lambda t}$ for some $\lambda >0$ but with a periodic prefactor $g_t$. Such oscillations, linked to the competition for the convergence to equilibrium between the position and the velocity (see the discussion p.66 of \cite{Desvillettes2003}), have also been numerically observed for the 
Boltzmann equation in \cite{Filbet}. This behaviour is reminescent of functions of the form $\phi(t)= e^{-\lambda t}\po a + b\cos(\nu t + \theta )\pf$, which are solutions of
\[\po(\partial_t + \lambda)^3 + \nu^2 (\partial_t + \lambda)\pf \phi = 0.\]
The third order may also be linked to the number of Lie Brackets one has to take in Hörmander's hypoellipticity theory to obtain a full rank (cf. \cite{Hormander1967}), and is expected to get bigger for higher order models (for instance oscillator chains \cite{Hairer2000}). Yet most of the current results rely on the existence of some quantity that somehow decreases \emph{at all time}, in other words in a first order differential equation (with the notable exception of \cite{Tran2012} where the usual dissipation of entropy is checked in mean in time). We can expect, in fact, a third order differential inequality to be satisfied, which can account for these inhomogeneities. This is the scope of the present article. This is not a new idea (cf. \cite{DesVilla2003}, \cite{Ledoux1995}) but up to our knowledge it had never been succesfully completed. In fact for the kFP model it has been noted in~\cite{Gadat2013} that no linear combination of the $L^2$ norm and its three first derivatives can be non-positive for 
all test 
functions, so we will clarify in the sequel the meaning of third order differential inequality.  

\subsection*{The models}

More precisely, this work will be devoted to study the relaxation to equilibrium for two kinetic models. The dynamics of the first one, the kFP process, is given by equation \eqref{equationkFP}. $(X_t,Y_t)\in\mathbb R^2$ is then the position-speed process of a particle in a potential $U$ with friction and noise. Results about its convergence to equilibrium can be found in \cite{Gadat2013} for a quadratic potential, and, according to one favorite method, \cite{Herau2005}, \cite{Guillin2006} or \cite{DesVilla2003} (among others) for more general cases (the coupling method, in \cite{Guillin2006}, only deals with convex potentials). 

The second one is a generalised version of the telegraph process, for which $(X_t,Y_t)\in\mathbb R\times \{\pm1\}$, where
$ dX_t  = Y_t dt$ and $Y_t$ jumps to its opposite following an inhomogeneous rate $a(X_t,Y_t)$. Here the particle go forward at constant speed and only does U-turn (cf. Figure \ref{figureA0} and \ref{figureA1} for an illustration). In the classical telegraph process the rate of jump $a$ is constant over its definition space. If we take $X_t \in \mathbb R/2\pi\mathbb Z$ to ensure ergodicity, we obtain maybe one of the simplest toy models for kinetic processes, cited as a basic example in \cite{Fellner2003} or \cite{DMS2011} and precisely studied in \cite{Volte-Face}. When the rate is no longer constant, the underlying algebra collapses. An ergodic version on the real line has recently been investigated in \cite{Fontbona2010} but, again with coupling method, the invariant measure corresponds to a convex potential.

\bigskip

In our cases, $(X_t,Y_t)$ has a unique invariant measure denoted $\mu$. Recall that the semi-group $(P_t)_{t\geq 0}$ of operators on $L^2( \mu)$ is defined by
\[P_tf(x,y) := \mathbb E\po f(X_t,Y_t)|X_0 = x,Y_0=y\pf.\]
Its infinitesimal generator $L$ is
\[L f := \overset{L^2(\mu)}{\underset{t\rightarrow 0}\lim} \frac{P_t f-f}{t} \]
for $f$ such that the limit exists. To focus on other questions, from now on we assume the existence of a core $\mathcal D$ dense in $L^2(\mu)$, stable by $L$, and we will always consider $f\in\mathcal D$. For a more analytical setting of the problem, denoting by $\hat L$ the dual of $L$, which operates on measures, the law $\mu_t$ of $(X_t,Y_t)$ is the (weak) solution of
\begin{displaymath}
 \left\{\begin{array}{l}
  \partial_t \mu_t = \hat L \mu_t\\
  \\
\mu_0 = law(X_0,Y_0).
 \end{array}\right.
\end{displaymath}

Then
\[P_tf(x,y) = \int f(u,v)\mu_t(du,dv)\]
when $\mu_0 = \delta_{(x,y)}$. We aim to quantify the convergence of $\mu_t$ to $\mu$.

\bigskip

For the kFP model, $\mu = e^{-U(x)}dx\otimes e^{-\frac{y^2}{2}}dy$ is the Gibbs measure associated to the Hamiltonian $U(x)+\frac{y^2}{2}$, and
\begin{eqnarray}\label{equationLdekFP}
 Lf & =& y\x f - U'(x)\y f  - y\y f+ \y^2f.
\end{eqnarray}
For the telegraph one, $\mu = e^{-U(x)}dx\otimes \frac{\delta_1+\delta_{-1}}2(dy)$ where $U'(x) = a(x,1)-a(x,-1)$ (see Lemma~\ref{lemMesureInvarCT}). Denoting $f_-(x,y) = f(x,-y)$,
\begin{eqnarray}\label{equationLdeCT}
 L f& =& y\x f + a(x,y) (f_--f).
\end{eqnarray}

\begin{figure}
\centering
\includegraphics[scale=0.3]{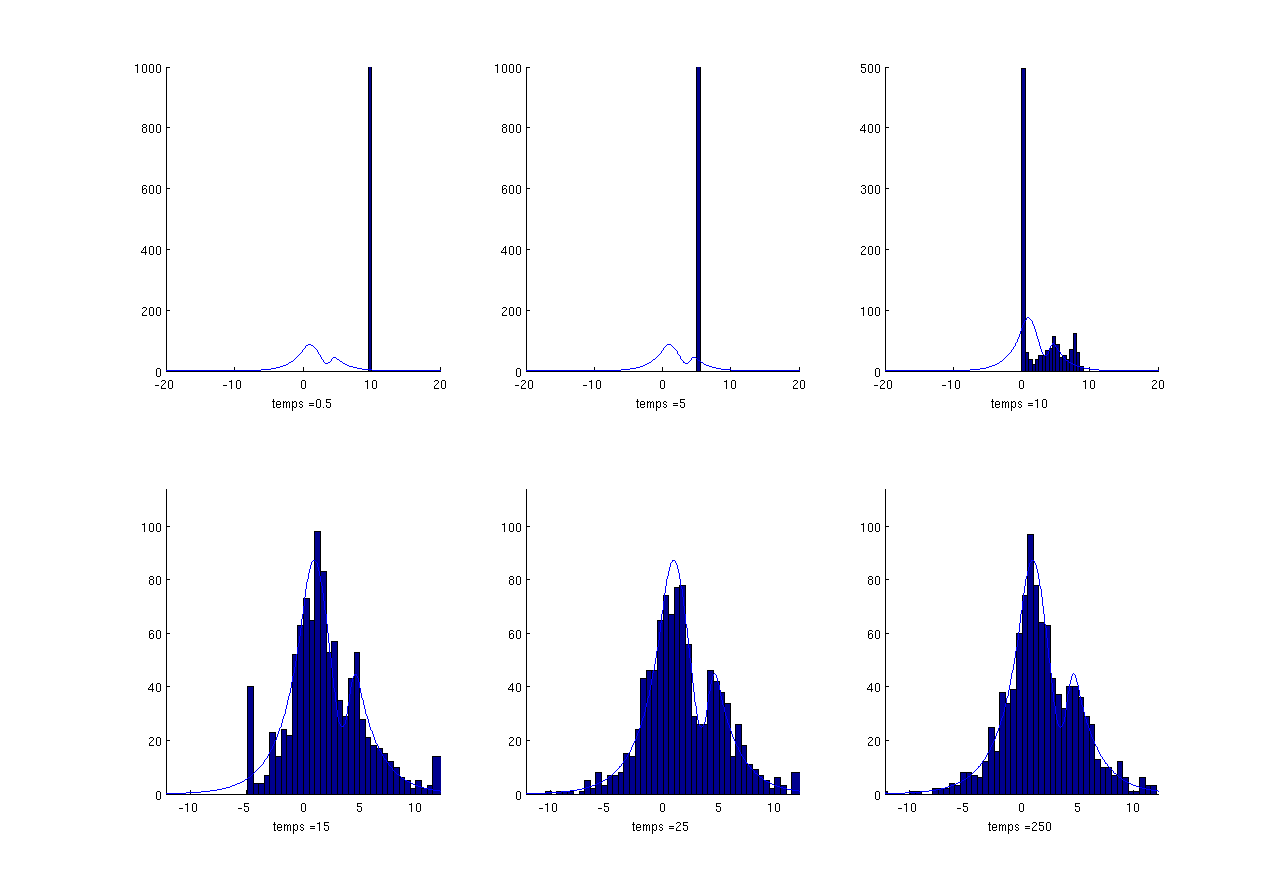}
 \caption{First marginal of the telegraph process at different time with a bi-modal invariant law $e^{-U(x)}dx$, $(X_0,Y_0) = (7,-1)$ and $a(x,y) = \po yU'(x)\pf_+$. While the potential decreases along the trajectory, the process is deterministic. It easily escapes  from the local minimum.}\label{figureA0}
\end{figure}

\begin{figure}
\centering
\includegraphics[scale=0.3]{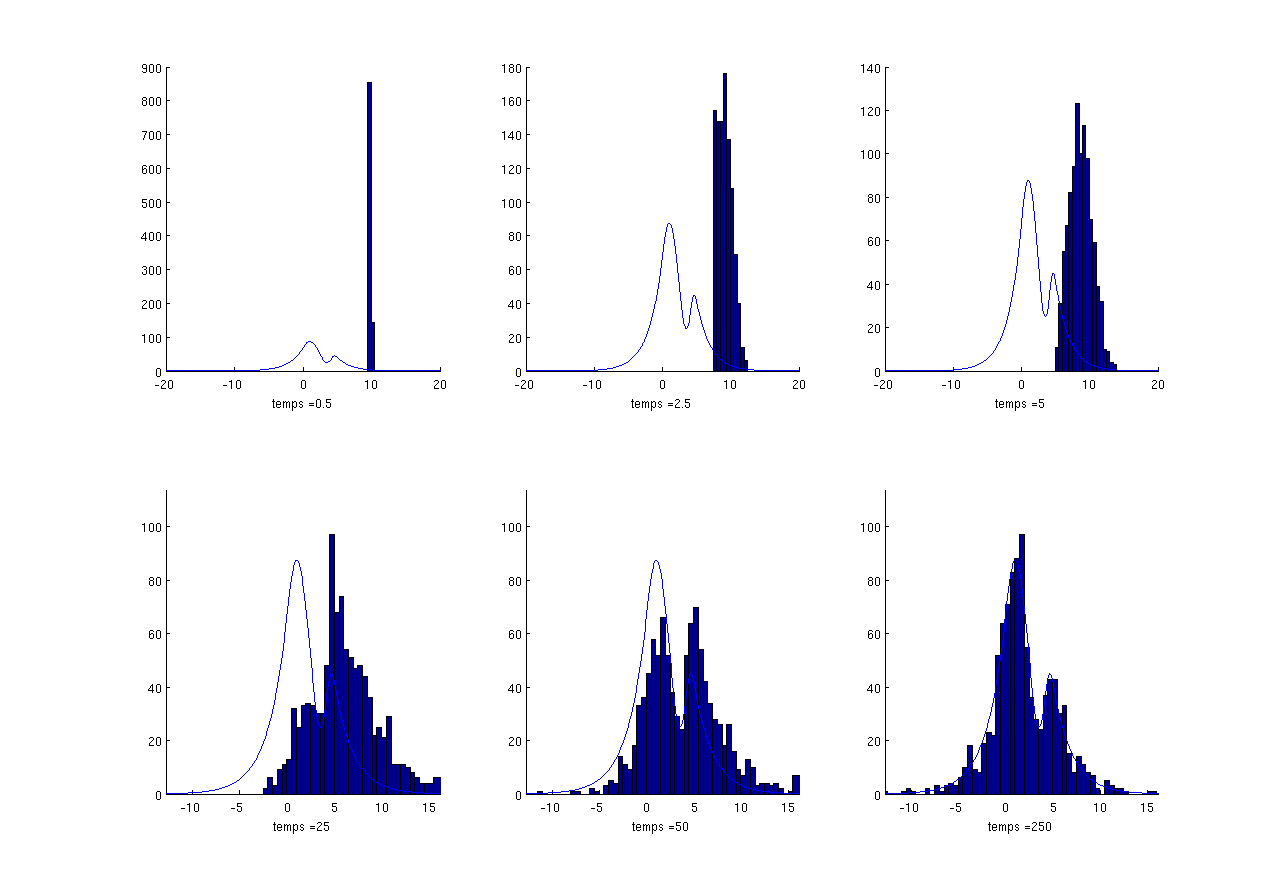}
 \caption{Here $a(x,y) = 1+\po yU'(x)\pf_+$. In other words, contrary to Figure \ref{figureA0}, there is always a minimal level of randomness : the behaviour is more diffusive and it takes longer to leave the local minimum.}\label{figureA1}
\end{figure}

These two processes share some common features. One of them is that there is no coercivity from the deterministic part of the dynamics when the potential is not convex; in other words two particles coupled with the same random part don't have any trend to get closer. In the other hand the randomness only occurs in the velocity variable, and thus the processes are fully degenerate in the sense of \cite{Cattiaux2008} and their Bakry-Emery curvature (definition 5.3.4 in \cite{logSob}) is equal to $-\infty$.

The study is restricted to dimension 1 in order to keep a reasonable level of computations and let the main ideas clear. The author did succesfully apply the method presented below to the telegraph in higher dimension, but surely we could improve our understanding of it and write it in more abstract settings, better suited for generalization.

\subsection*{Main result}

Let $f_t = (P_t - \mu) f$ where $\mu f = \int f d\mu$, so that $\forall t\geq0$, $\mu f_t = 0$; in other words $f_t \in 1^\perp$ the orthogonal space of the constants in $L^2(\mu)$. Let $F_t =\|f_t\|^2_{L^2(\mu)}$. In the following (cf. Section \ref{SectionInegalite}) we show that, under some assumptions on the potential $U$ or on the rate $a$, there exist explicit $\lambda,\eta>0$ and $\nu_*\in\mathbb R$ and a function $t\mapsto \nu_t\geq \nu_*$ such that
\begin{eqnarray}\label{equationORdre3}
 (\partial_t + \lambda)\Big[ (\partial_t + \eta)^2 + \nu_t\Big]F_t & \leq & 0.
\end{eqnarray}
Furthermore $\nu_*$ is such that the roots of the polynomial $(X+\eta)^2+\nu_*$ have negative real part, namely either $\nu_* \geq 0$ or $\sqrt{|\nu_*|} < \eta$ ; to sum up, $\mathcal Re(\eta - \sqrt{-\nu_*})>0$.
Then exponential decay follows from the next result.
\begin{thm}\label{propconvergence}
Assume \eqref{equationORdre3} holds.
\begin{itemize}
 \item if $\nu_* \leq 0$ then $F_t \leq \phi_t$
\item if $\nu_* > 0$ then $F_t \leq \phi_t+ e^{-\eta t} \underset{s\leq t}\sup\ e^{\eta s}\po\phi_s-F_s\pf$ ; furthermore the length of a time interval where $F > \phi$ is less than $\frac{\pi}{\nu_*}$
\end{itemize}
with in both cases $\phi$ solution of
\[(\partial_t + \lambda)\Big[ (\partial_t + \eta)^2 + \nu_*\Big]\phi_t = 0,\]
$\phi_0 = F_0$, $\phi'_0 = F'_0$ and an explicit $\phi''_0$.
In particular,
\[\underset{t\rightarrow \infty}{\lim} \frac1t\ln F_t \leq -\min\po\lambda,\mathcal Re(\eta - \sqrt{-\nu_*})\pf\]
\end{thm}

\begin{rmq}
\begin{itemize}
\item $\nu_* \leq 0$ can always be assumed.
 \item For $\nu_*>0$, $\phi$ presents damped oscillations with a period $\frac{2\pi}{\nu_*}$ and a magnitude of order $e^{-\eta t}$. The theorem shows that $F$ is interlaced with $\phi$ : $F$ can be above $\phi$ but only if it's already been below, and not for too long.
 \item The rate of convergence is independent from the function $f$, but this result does not give a bound for the operator norm of the semi-group in $L^2$. As will be seen in the sequel, $\phi''_0$ depends on $F''(0)$ and $\|\x f_0\|^2$, which can be arbitrarily large with $F_0 = 1$ (we could obtain a bound by using estimates from the pseudodifferential calculus theory, but our aim was to avoid resorting to this powerfull tool and to stay very elementary). The result in \cite{DMS2009} does the job with no derivative - but not exactly with the $L^2$ norm ; it could be possible to do the same in the present work.
\end{itemize}
\end{rmq}

\begin{proof}
The Gronwall lemma gives
\[\Big[ (\partial_t + \eta)^2 + \nu_*\Big]F_t\leq \Big[ (\partial_t + \eta)^2 + \nu_t\Big]F_t \leq \left(\big[ (\partial_t + \eta)^2 + \nu_t\big]F_t\right)_{t=0}e^{-\lambda t} := C_0e^{-\lambda t}.\]
Let $\phi$ be the solution of
\begin{displaymath}
 \left\{\begin{array}{c}
         \Big[ (\partial_t + \eta)^2 + \nu_*\Big]\phi =  C_0e^{-\lambda t}\label{phi}\\
\\
\phi_0 = F_0,\ \phi'_0 = F'_0.
        \end{array}\right.
\end{displaymath}
Thus
\[\Big[ (\partial_t + \eta)^2 + \nu_*\Big](F_t - \phi_t) \leq 0.\]
In the case where $\nu_* \leq 0$, using twice the Gronwall lemma gives
\begin{eqnarray*}
 (\partial_t + \eta - \sqrt{-\nu_*})(\partial_t + \eta + \sqrt{-\nu_*})(F_t-\phi_t) & \leq & 0\\
 \Rightarrow\hspace{70pt}(\partial_t + \eta + \sqrt{-\nu_*})(F_t-\phi_t) & \leq & 0\\
 \Rightarrow\hspace{143.5pt}F_t - \phi_t & \leq & 0.
\end{eqnarray*}
So now assume $\nu_* > 0$ and define $h_t = e^{\eta t}\po F_t - \phi_t\pf$, so that $h_t''+\nu_* h_t \leq 0$, $h_0 = h'_0 = 0$. Define $M_t = \sup\ \{-h_s,\ 0 \leq s\leq t\}$ so that $h_t \geq -M_t$. $M_t$ is always nondecreasing and it is constant when $h$ is increasing. Le us show that $h_t \leq M_t$ at every time. Assume it is false and consider $s = \inf\{t>0,\ h_t > M_t\}$. $M_t$ is constant for $t$ in a neighborhood of $s$, $h_{s-\varepsilon} < M_s$ and $h_{s+\epsilon} > M_s$ for $\varepsilon >0$ small enough. So, as $h''_s \leq -\nu_* h_s<0$, necessarily $h'_s > 0$, which leads to $(h'_s)^2 + \nu_*h_s^2 > \nu_*M_s^2$. Now consider $u = \sup\{0 \leq t\leq s, h'_t \leq 0\}$ (which exists if $s$ exists, as $h''_0 \leq 0$) and note that $M_u = M_s$. We get $(h'_u)^2 + \nu_*h_u^2 = \nu_*h_u^2 \leq \nu_*M_u^2$. Yet for $t \in(u,s]$, $h_t' > 0$ so
\[h_t'\po h_t''+\nu_* h_t\pf = \frac12\frac{d}{dt}\po(h'_t)^2 + \nu_* h_t^2\pf \leq 0\]
and we've reached a contradiction.

Concerning the length of an interval where $F > \phi$, in other words where $h > 0$, define on this interval $\delta_t = -\frac{1}{h_t}\po h_t''+\nu_* h_t\pf\geq 0$. Thus $h_t$ is solution of
\[\psi_t'' + (\nu_* + \delta_t) \psi_t = 0\]
and so vanishes, according to the Sturm-Liouville comparison theorem (cf. \cite{Dieudonne} for instance), between two successive zeros of $\cos(\nu_* t + \theta)$ for any $\theta$.
\end{proof}

In Section~\ref{SectionInegalite}, the kFP and telegraph models are proven to satisfy an inequality of the form \eqref{equationORdre3}. Section \ref{numerique} is devoted to numerical studies, whose conclusion is that the method can give the good order of magnitude for the exponential rate of convergence, but shouldn't be trusted to compute parameters which accurately give the asymptotically fastest convergence. Finally an appendix gathers the proof of the technical lemmas used throughout this work.

\bigskip 

\textbf{Acknowledgements.} The author thanks Laurent Miclo, who initiated this work, and Sebastien Gadat, for fruitfull discussions.

\section{Third order inequality}\label{SectionInegalite}

We start with considerations applying to both models. To compute the derivatives of $F_t$, we'll split $L$ in its symmetric and anti-symmetric part. More precisely, if $A$ and $B$ are operators on $L^2(\mu)$, we denote by $A^*$ the dual operator of $A$ and by $[A,B]$ the Lie Brackets $AB-BA$. $<,>$ stands for the scalar product on $L^2(\mu)$.
\begin{lem}\label{lemDerivesdeFt}
Assume
\[L = K + R - R^*\]
with $K^* = K$. Then
\begin{eqnarray*}
 F_t' & = &  <(2K)f_t,f_t>\\
F_t'' & = &  <(2K)^2f_t,f_t> + 4 <[K,R]f_t,f_t>\\
F_t''' & = &  <(2K)^3f_t,f_t> + 12  <[K^2,R]f_t,f_t> + 4 <\left[[K,R],R-R^*\right] f_t,f_t>.
\end{eqnarray*}
\end{lem}
The proof is given in the appendix. The successive derivation of $F_t$ could also be obtained with iterated $\Gamma$-calculus (see \cite{Ledoux1995}), in particular for models where the invariant measure is not so easy to handle.

\bigskip

As in kinetic models the coercive part $K$  of $L$ only acts on the velocity variable, one cannot find any  $\lambda >0$ such that, for all $f_t$, $F_t' \leq -\lambda F_t$. We call $\mu_1$ (resp. $\mu_2$) the first (resp. second) marginal of $\mu$, namely the position (resp. velocity) distribution at equilibrium. In our specific models we'll have $\mu = \mu_1 \otimes \mu_2$. We call $V = Ker (\mu_2-1)$ the set of functions which do not depend on $y$. The orthogonal projection to $V$ and $V^\perp$ will be respectively denoted by $\pi_V$ and $\pi_\perp$:
\[\pi_V f(x,y) = \mathbb E\po f(x,Z)|Z\sim \mu_2\pf = \int f(x,z)\mu_2(dz),\hspace{20pt}\pi_\perp = 1 -\pi_V.\]
We will note $f_V = \pi_V f_t$ and $f_\perp = \pi_\perp f_t$; as $f_V$ only depends on $x$ we will sometimes consider $f_V$ as a one-parameter function in $L^2(\mu_1)$. Finally let $G_t = \|\x f_t\|^2$, and  recall that a measure $\nu$ is said to satisfy a Poincaré (or spectral gap) inequality with constant $c$ if
\begin{eqnarray}
 \int |\partial_z g(z)|^2d\nu(z) \geq c \int |g(z)|^2d\nu(z)
\end{eqnarray}

whenever $\nu g = 0$.
\begin{lem}\label{GtFtcoercif}
We have $\mu_1 f_V = 0$. In particular, if $\mu_1$ satisfies a Poincaré inequality with constant $c$, we have
 \[G_t \geq \|\x f_V\|^2 \geq c \|f_V\|^2 \]
If furthermore $F_t' \leq - d\|f_\perp\|^2$ then
\begin{eqnarray*}
  \frac1{F_t}\po G_t - \frac{c}{d}F_t'\pf&\geq& c
\end{eqnarray*}

\end{lem}
\begin{proof}
 For the first assertion,
 \[\int f_V(x) d\mu_1(x)  =  \int \po \int f_t(x,y)d\mu_2(y)\pf d\mu_1(x) = \mu f_t = 0\]
 Furthermore $\x^*\x$ is self-ajoint and stabilizes $V$, so it stabilizes $V^\perp$ and
 \[G_t = \|\x (f_V + f_\perp)\|^2 = \|\x f_V\|^2 + \| \x f_\perp\|^2 \geq \|\x f_V\|^2.\]
 Then $G_t - \frac{c}{d}F_t' \geq cF_t$ is clear.
\end{proof}

Now we will show that in both models, the inequality \eqref{equationORdre3} holds for some parameters.

\subsection{The kinetic Fokker-Planck process}

In this section (from Lemma \ref{lemFPderives} to Theorem \ref{thmFPracinecommune}) the generator is
\[L  = y\x - U'(x)\y  - y\y+ \y^2 \tag{\ref{equationLdekFP}}\]

The invariant measure is $\mu  =  e^{-U(x)}dx \otimes e^{-\frac{y^2}2}dy$ so that \[\x^* =  U' - \x, \hspace{20pt}\y^*  =  y -\y.\]
From now on we will make some assumptions on the potential $U$ :
\begin{hyp}\label{U}
 The potential $U$ is smooth, $U''$ is bounded and $\mu_1 = e^{-U(x)}dx$ satisfies a Poincaré inequality with constant $c_U$
 \end{hyp}
The smoothness and the Poincaré inequality conditions are usual assumptions (for instance in \cite{Villani2009}, \cite{DMS2009}) ; however the boundedness of $U''$ is quite restrictive, and could be an artefact due to the lack of subtility of some of our computations.

We can decompose $L = K + R - R^*$ with
\[K=- \y^*\y\hspace{20pt}R=-\x^*\y\hspace{20pt}R^*=\hspace{20pt}-\y^*\x.\]
We compute in appendix the brackets appearing in Lemma \ref{lemDerivesdeFt} :
\begin{lem}\label{lemFPderives}
\begin{eqnarray*}
 \left[K,R\right] & = &  R\\
 \left[K^2,R\right] & = & R\left(1 +2 K\right)\\
 \left[R,R^*\right] & = & U''K + \x^*\x. 
\end{eqnarray*} 
\end{lem}

As expected the operator $-\x^*\x$ appears in the third derivative:
\[F_t''' = <(2K)^3 f_t, f_t > + 12 <R(1+2K) f_t, f_t> - 4 <\x f_t, \x f_t> + 4 <U''(x)\y f_t,\y f_t>.\]
It brings the coercivity in position, which is missing in $F_t'$. However it is known (cf. \cite{Gadat2013}, \cite{Volte-Face}) that no linear combination  of $F_t$, $F_t'$, $F_t''$ and $F_t'''$ can be non-positive for every ${f\in L^2(\mu)}$. In the particular cases treated in \cite{Gadat2013} and \cite{Volte-Face}, $G_t$ the norm of the gradient in space appears naturally, thanks to Lemma \ref{GtFtcoercif}. Indeed the smaller eigenvalue of $\y^*\y$ on ${V^\perp = (Ker\y^*\y)^\perp}$ is 1 (Poincaré inequality for the gaussian distribution) and thus
\[ F_t' = 2<K f_t, f_t> = -2\|\y f_\perp\|^2\leq -2 \| f_\perp \|^2.\]
In the other hand, Assumption \ref{U} and Lemma \ref{GtFtcoercif} ensure $G_t \geq c_U \| f_V\|^2$ and lead to
\begin{eqnarray}\label{equationDefinition_gt}
 \frac{1}{F_t}\po G_t - \frac{c_U}{2} F_t'  \pf & \geq&  c_U.
\end{eqnarray}


Finally in order to close the differential inequality we need the first derivative of $G_t$ (see Appendix for the proof):
\begin{lem}\label{lemDerivedeGt_kFP}
 \begin{eqnarray*}
 G_t' & =&  <\left(-2 RR^* + 2RU''\right)f_t,f_t> + 2 G_t.
\end{eqnarray*}
\end{lem}

We can now find a linear combination of $F_t$, $G_t$ and their derivatives which is always non-positive. The terms in $\|f_V\|^2$ will be controlled by $F_t'''$, the ones in $\|f_\perp\|^2$ by $F_t'$ and $G_t'$.

\begin{lem}\label{lemFPinegalite}
 Let $A \in\mathbb R$ and $\beta,k >0$. Under Assumption \ref{U}, there exists $\tau_*\in\mathbb R$ such that for all $\tau \geq \tau_*$
\[Q_3(\partial_t) F_t + Q_1(\partial_t) G_t \leq 0\]
with
\begin{eqnarray*}
Q_1(X) & =& 2\beta\po X + 2(\frac1\beta-1)-\frac k{2c_U\beta}\pf\\ 
\\
Q_3(X) & = & X^3 + AX^2 + \tau X + k.
\end{eqnarray*}
\end{lem}
\begin{proof}
The above computations (Lemma \ref{lemDerivesdeFt} and \ref{lemFPderives}) allow to write, for any $A\in\mathbb R$,
 \begin{eqnarray*}
 F_t''' + A F_t'' & = & <\po(2K)^3 +A(2K)^2 + 4\left[[K,R],R-R^*\right] + 12[K^2,R] + 4A [K,R] \pf f_t ,f_t>\\
 & = & <\po(2K)^3 +A(2K)^2 - 4\left[R,R^*\right] + 12R(1+2K) + 4A R \pf f_t ,f_t>\\
 & = & <\po(2K)^3 +A(2K)^2 -4U''K -4\x^*\x+ 4R\po3+A+6K\pf\pf f_t,f_t> \\
 & = & <\po(2K)^3 +A(2K)^2 -4U''K + 4R\po3+A+6K\pf\pf f_t,f_t> - 4 G_t
\end{eqnarray*}
The operator $R\po6+6(2K)+2A\pf$ is annoying because, as a quadratic form on $L^2(\mu)$, it is neither positive nor non-positive ; we'll give for it a not so subtle upper bound by the Cauchy-Schwarz and $2ab \leq a^2 + b^2$ inequalities with the sum of a term $RR^*$ to be controlled by $G_t'$ and of a term only acting on $V^\perp$, controled by $F_t'$.

More precisely, remark that $R=R\pi_\perp$ and furthermore that $\pi_\perp$ commutes with the self-ajoint operators $Id$, $K$ and $U''(x)$ which stabilize $V$ (and so $V^\perp$ too). Thus, for any $\beta >0$,
\begin{eqnarray*}
 & & <\po4R\pi_\perp\po3+A+\beta U''+6K\pf\pf f_t,f_t>  \\
\\
& =& 2<\po3+A+\beta U''+6K\pf f_\perp , (2R^*)f_t>  \\
 \\
&\leq & 4\beta <RR^* f_t,f_t> + \frac1\beta<\po3+A+\beta U''+6K\pf^2 f_\perp,f_\perp>.
\end{eqnarray*}
We obtain, taking into account lemma \ref{lemDerivedeGt_kFP},
\begin{eqnarray*}
\lefteqn{F_t''' + A F_t'' + 2\beta G_t'+4(1-\beta)G_t =}\\
\\
&  & <\po(2K)^3 +A(2K)^2 -4U''K + 4R\po3+A+6K\pf - 4\beta R R^* +4 \beta RU'' \pf f_t,f_t>\\
\\
& = & <\po(2K)^3 +A(2K)^2 -4U''K  - 4\beta R R^*  \pf f_t,f_t> + <4R\po3+A+\beta U''+ 6K\pf f_t,f_t>\\
\\
& \leq & <\Big((2K)^3 +A(2K)^2 -4U''K+ \frac1\beta\po3+A+\beta U''+6K\pf^2 \Big) f_\perp,f_\perp>\\
\\
& = & <\Big((2K)^3 +\po A+\frac9\beta\pf (2K)^2 +\left(\frac{18+6A}{\beta}+4 U''\right)(2K) + \frac1\beta(3+A+\beta U'')^2\Big) f_\perp,f_\perp>.
\end{eqnarray*}
Now we want to replace the terms with $U''$ by something that does not depend on $x$ (under Assumption \ref{U}).
\begin{eqnarray*}
 <U'' K f_t,f_t> & = & -\iint U''(x) \po\y f_t(x,y)\pf^2 \mu(dx,dy) \\
 & \leq & <(\min U'')K f_t,f_t>
\end{eqnarray*}
and
\begin{eqnarray*}
 <(3+A+\beta U'')^2 f_\perp,f_\perp> & \leq & \| 3+A+\beta U'' \|_\infty^2 \|f_\perp\|^2.
\end{eqnarray*}
So by denoting
\[P(X) = X^3 + \po A+\frac9\beta\pf X^2 + \left(\frac{18+6A}{\beta}+4\min U''\right) X + \frac1\beta\| 3+A+\beta U'' \|_\infty^2,\]
the previous computation leads to
\begin{eqnarray*}
 F_t''' + A F_t'' + 2\beta G_t'+4(1-\beta)G_t & \leq & <P(2K) f_\perp, f_\perp>.
\end{eqnarray*}
The eigenvalues of $2K$ on $V^\perp$ being the $-2n$ for $n\in\mathbb Z_+$ (the eigenvectors are the so-called Hermite polynomials), consider any $k\geq0$ and
\begin{eqnarray*}
 \tau_k & =&  \underset{n\geq 1}{\max} \frac{P(-2n)+k}{2n} \hspace{20pt}(<\infty)
\end{eqnarray*}
so that $P(2K) + 2\tau K + k$ gets to be a non-positive bilinear form on $V^\perp$ for all $\tau \geq \tau_k$, in other words
\begin{eqnarray*}
 <P(2K)f_\perp,f_\perp > & \leq  &-\tau F_t' - k \|f_\perp\|^2.
\end{eqnarray*}
On the other hand $G_t \geq c_U \|f_V\|^2$ (cf. Lemma \ref{GtFtcoercif}) so that
 \begin{eqnarray*}
F_t''' + A F_t'' + 2\beta G_t' + \left(4(1-\beta)-\frac{k}{c_U}\right)G_t + \tau F_t' + kF_t& \leq & 0.
\end{eqnarray*}
\end{proof}
Now it remains to get rid of $G_t$ thanks to \eqref{equationDefinition_gt}, and to find a common root for $Q_1$ and $Q_3$ in order for inequation \eqref{equationORdre3} to hold.
\begin{thm}\label{thmFPracinecommune}
Under assumption \ref{U}, there exist $\lambda,\eta>0$, and $t\mapsto \nu_t\geq \nu_*\in\mathbb R$ with ${\mathcal Re(\eta - \sqrt{-\nu_*})>0}$ such that
\[(\partial_t + \lambda)\Big[ (\partial_t + \eta)^2 + \nu_t\Big]F_t\leq 0.\] 
\end{thm}
\begin{proof}
Let $A \in\mathbb R$ and $\beta\in(0,1]$. Let $k \in[0,4c_U(1-\beta)]$, so that the root of
\[ Q_1(X) = 2\beta\po X + 2(\frac1\beta-1)-\frac k{2c_U\beta}\pf\]
is zero for $k=4c_U(1-\beta)$ and negative otherwise. Let $\tau_k$ be given by Lemma \ref{lemFPinegalite}; we choose $\tau\geq\tau_{4c_U(1-\beta)}$ large enough such that for all $k\geq 0$ 
\[Q_3(X) = X^3 + AX^2 + \tau X + k\]
has only one non-positive root, which is continuous with respect to  $k$. This root is zero for $k=0$, negative otherwise. Thus by continuity there exists a $k\in[0,4c_U(1-\beta)]$ such that $Q_1$ and $Q_3$ have a common root. We call this root $-\lambda$. Now Lemma \ref{lemFPinegalite} can be rewritten, with some constant $u,v,w\in\mathbb R$,
\begin{eqnarray*}\label{equationlocale}
 0 & \geq & (\partial_t + \lambda) \po F_t'' + u F_t' + v F_t + w G_t \pf\\
 & = & (\partial_t + \lambda) \po F_t'' + \po u + w\frac{c_U}{2}\pf F_t' + \po v + w\po G_t - \frac{c_U}{2}F_t'\pf \frac 1{F_t}\pf F_t \pf.
\end{eqnarray*}
Let $\eta =  \frac{2u+wc_U}{4}$, $\nu_* = v + w c_U - \eta^2$ and
\begin{eqnarray*}
\nu_t& = & \po v + w\po G_t - \frac{c_U}{2}F_t'\pf \frac 1{F_t}\pf - \eta^2.
\end{eqnarray*}
Inequation \eqref{equationDefinition_gt} exactly means $\nu_t \geq \nu_*$. It remains to show that, for some parameters, $-\lambda$ and $\mathcal Re(-\eta \pm \sqrt{-\nu_*})$ are negative. These are the real parts of the roots of
\begin{eqnarray*}
 Q_3(X) +\frac {c_U}2 X Q_1(X) + c_U Q_1(X)
& =& X^3 + (A+\beta c_U)X^2 + (\tau + 2c_U -\frac k 2) X + 2c_U(1-\beta).
\end{eqnarray*}
Take $\beta = 1$, $A > - c_U$, $\tau$ large enough and $k = 0$. Then zero is a common root and
\begin{eqnarray*}
 Q_3(X) +\frac {c_U}2 X Q_1(X) + c_U Q_1(X) & =& X\po X^2 + ( A+\beta c_U)X + (\tau + 2c_U)\pf,
\end{eqnarray*}
so that $\lambda = 0$. $X^2 + ( A+\beta c_U)X + (\tau + 2c_U)$ has positive coefficients so if it has real roots, they are negative. Otherwise  ${\mathcal Re(\eta - \sqrt{-\nu_*}) = \eta = \frac 12 ( A+\beta c_U) > 0}$. Now if $\beta$ is chosen slightly less than 1 and $k$ is such that $Q_1$ and $Q_3$ still have a common root, relying again on continuity, we still have $\mathcal Re(\eta - \sqrt{-\nu_*}) >0$ but $-\lambda$ becomes a real root of a polynomial with positive coefficients and thus is negative.
\end{proof}

\subsection{The telegraph process}

This section is a replica of the previous one. From Lemma \ref{lemMesureInvarCT} to Theorem \ref{thmCTracinecommune}, the generator is
\[L f(x,y) = y\x f (x,y) + a(x,y) \po f(x,-y) - f(x,y)\pf \tag{\ref{equationLdeCT}}.\]
As in the kFP case we compute the derivatives of $F_t$ and $G_t$, proceed with a differential equation and conclude with a particular choice of the parameters which are in parties to the above approach. First, the invariant measure has to be explicited:

\begin{lem}\label{lemMesureInvarCT}
 The unique (up to a constant) invariant measure of the telegraph model is \[\mu = e^{-U(x)}dx\otimes \frac12\po\delta_1+\delta_{-1}\pf (dy),\] where
 \[U'(x) = a(x,1)-a(x,-1).\]
\end{lem}

\begin{proof}
 Note that $yU'(x)=a(x,y)-a(x,-y)$. We check
 \begin{eqnarray*}
  \mu L f & = & \iint y\x f(x,y)e^{-U(x)}dx dy + \iint a(x,y) \po f(x,-y) - f(x,y)\pf e^{-U(x)}dx dy\\
  & = & \iint y f(x,y) U'(x)e^{-U(x)}dx dy + \iint \po a(x,-y) - a(x,y) \pf f(x,y) e^{-U(x)}dx dy\\
  & = & 0
 \end{eqnarray*}
for all smooth $f\in L^2(\mu)$, so that $\mu$ is invariant. In the other hand, the process is clearly irreducible (from any point $X_0$ it can reach any ball in finite time with positive probability) and aperiodic ($X_t$ can go back to $X_0$ at an arbitrarily small time $s$ with positive probability) and uniqueness of its invariant probability follows.
\end{proof}
We note $f_-(x,y) = f(x,-y)$ and remark  that \[V = \{f \in L^2(\mu), f_- = f\}\hspace{20pt}V^\perp = \{f \in L^2(\mu), f_- = -f\}\] and, keeping the previous notation $\pi_V$ and $\pi_\perp$ (or $f_V$ and $f_\perp$) for the orthogonal projections on $V$ and $V^\perp$,
\[\pi_\perp f = \frac{f-f_-}{2}  \hspace{20pt} \pi_V f =\frac{f+f_-}{2} = f_V.\] 
Thus $yV = V^\perp$ and $yV^\perp = V$, and more precisely
\[\pi_\perp y = y\pi_V \hspace{20pt} \pi_V y = y\pi_\perp.\] 
Now recall that $\x^* = \x - U'$ and define
\begin{eqnarray*}
 Kf & =& -(a+a_-)\pi_\perp f\\
Rf & = & y \po\x - U'\pf\pi_\perp f\\
& = &  -\pi_V y \x^* \pi_\perp f.\\
R^* f& = & -\pi_\perp y\x \pi_V f\\
& = & -y\x\po\frac{f + f_-}{2}\pf.
\end{eqnarray*}
Then
\begin{eqnarray*}
 K + R - R^* & = & -(a+a_-)\pi_\perp + y (\x-U') \pi_\perp  + y \x \pi_V \\
 & = & \po-a - a_- - yU'\pf \pi_\perp + y\x (\pi_V + \pi_\perp)\\
 & = & -2a\pi_\perp + y\x\\
 & = & L.
\end{eqnarray*}
Note that $a+a_-$ and $U'$ do not depend on $y$ and so, seen as self-adjoint operators on $L^2(\mu)$, they commute with $\pi_\perp$ and $\pi_V$. In particular this gives $K^* = K$.

Now thanks to these considerations we can compute the following brackets (see the appendix for details).
\begin{lem}\label{lemCTFtderives}
 \begin{eqnarray*}
 [K,R] & = & R(a+a_-)\\
\left[K^2,R\right] & = & -R(a+a_-)^2\\
\left[\left[K,R\right],R-R^*\right] & = & R^*R(a+a_-) - R(a+a_-)R^*\\
RR^* & = & \x^*\x \pi_V\\
R^*R & = & (\x^*\x + U'')\pi_\perp.
\end{eqnarray*}
\end{lem}
From now on we make the following assumptions:
\begin{hyp}\label{a}
The rates $a(.,1)$ and $a(.,-1)$ are positive, smooth, bounded and with bounded derivatives, and
 \begin{eqnarray*}
  a_* & := & \underset{\mathbb R}\min\po a(.,1)+a(.,-1)\pf > 0.
 \end{eqnarray*}
 Furthermore $U(x)=\int_0^x\po a(z,1)-a(z,-1)\pf dz$ satisfies Assumption \ref{U}.
\end{hyp}
Then, again $F_t$ is controlled by $G_t$ and $F_t'$. Indeed, Under Assumption~\ref{a},
\[F_t' = 2<Kf_t,f_t> = -2<(a+a_-)f_\perp,f_\perp> \leq -2 a_* \|f_\perp\|^2 \]
and (Lemma \ref{GtFtcoercif}),
\[G_t = \| \x f\| ^2 \geq \| \x f_V\| ^2 \geq c_U\|f_V\|^2 \]
so that 
\begin{eqnarray}\label{definitiongtCT}
 \frac{1}{F_t}\po G_t - \frac{c_U}{2 a_*} F_t' \pf & \geq & c_U.
\end{eqnarray}

We will also need the derivative of $G_t$, computed in the appendix:
\begin{lem}\label{CTGtderive}
 \[G_t' = <2\po RU'' -U''K -R^*R(a+a_-)\pf f_t,f_t>.\]
\end{lem}
We are now ready to prove a result similar to Lemma \ref{lemFPinegalite} 

\begin{lem}\label{lemCTinegalite}
 Under Assumption \ref{a}, there exist polynomials $\tilde Q_1$ and $\tilde Q_3$ respectively of first and third order such that
 \[\tilde Q_3(\partial_t )F_t + \tilde Q_1(\partial_t)G_t \leq 0\]
\end{lem}
\begin{proof}
Lemma \ref{lemDerivesdeFt} and \ref{lemCTFtderives} give, for any $A\in\mathbb R$,
\begin{eqnarray*}
\lefteqn{ F_t''' + A F_t'' =}\\
  & & <\big((2K)^3+A(2K)^2 + 4R\po-3(a+a_-)^2 + A(a+a_-)\pf + 4 R^* R(a+a_-) - 4 R(a+a_-)R^*\big) f_t,f_t>.
\end{eqnarray*}
We consider any $h \geq 0$ and write
\begin{eqnarray*}
  F_t''' + A F_t'' +2(1+h)G_t'
  & = & <\big((2K)^3+A(2K)^2- 4(1+h)U''K-4h R^* R(a+a_-)  - 4 R(a+a_-)R^*\big) f_t,f_t>\\
  & & + <4R\po-3(a+a_-)^2 + A(a+a_-)+(1+h)U''\pf  f_t,f_t>
\end{eqnarray*}
Now for the extra $-4h R^* R(a+a_-)$, for any $\alpha\in(0,1]$, \emph{via} the Cauchy-Schwarz inequality,
\begin{eqnarray*}
 & & -4h <R^* R(a+a_-)f_t,f_t> \\
 & = & -4h <\x (a+a_-) f_\perp,\x f_\perp>\\
 & = & -4h<(a+a_-)\x f_\perp,\x f_\perp> - 4h <\po\x(a+a_-)\pf f_\perp,\x f_\perp>\\
  & = & -4h<(a+a_-)\x f_\perp,\x f_\perp> - 4h <(a+a_-)^{-\frac{1}{2}}\po\x(a+a_-)\pf f_\perp,(a+a_-)^{\frac{1}{2}}\x f_\perp>\\
 & \leq & -4ha_*(1 - \alpha) \|\x f_\perp \|^2 + \frac h\alpha \| (a+a_-)^{-\frac{1}{2}}\po\x(a+a_-)\pf f_\perp \|^2.
\end{eqnarray*}
Then following again the steps of Lemma \ref{lemFPinegalite}, for any $\beta >0$, we bound
\begin{eqnarray*}
 & & <4R\po-3(a+a_-)^2 + A(a+a_-)+(1+h)U''\pf f_t, f_t >\\
 &= & 2<\po-3(a+a_-)^2 + A(a+a_-)+(1+h)U''\pf(a+a_-)^{-\frac12} f_\perp, 2(a+a_-)^{\frac12}R^* f_t>\\
 & \leq &  \frac{1}{\beta}\|\po-3(a+a_-)^2 + A(a+a_-)+(1+h)U''\pf(a+a_-)^{-\frac12} f_\perp\|^2 + 4\beta <R(a+a_-)R^* f_t,f_t>
\end{eqnarray*}
Gathering all this, and recalling $K = -(a+a_-)\pi_\perp$ we get
\begin{eqnarray*}
\lefteqn{ F_t''' + A F_t'' +2(1+h)G_t' \leq}\\
  & & <\po-8(a+a_-)^3+4A(a+a_-)^2 + \frac{\po-3(a+a_-)^2 + A(a+a_-)+(1+h)U''\pf^2}{\beta (a+a_-)}+4(1+h)U''(a+a_-)\pf f_\perp,f_\perp>\\
  & & -4ha_*(1 - \alpha) \|\x f_\perp \|^2 + \frac h\alpha \| \frac{\x(a+a_-)}{\sqrt{a+a_-}} f_\perp \|^2 - 4 <R(a+a_-)(1 -\beta)R^* f_t,f_t>.
\end{eqnarray*}
For the last term, as long as $\beta \leq 1$,
\begin{eqnarray*}
 -4 <R(a+a_-)(1 -\beta)R^* f_t,f_t> & = & -4 <(a+a_-)(1 -\beta) \x f_V,\x f_V>\\
 & \leq & -4 a_*(1 -\beta) \| \x f_V\|^2.
\end{eqnarray*}
Choose $\alpha < 1$, let $k\in\big[0,4c_U(1-\beta)\big]$ and define $h_k$ such that
\[-4h_k\po 1-\alpha \pf = -4(1 - \beta) + \frac{k}{a_*c_U}.\]
We will note
\[\lambda_k = a_*\frac{2h_k\po 1-\alpha \pf}{(1+h_k)}\]
and define the function
\[H = -8(a+a_-)^3+4A(a+a_-)^2 + \frac{\po-3(a+a_-)^2 + A(a+a_-)+(1+h_k)U''\pf^2}{\beta (a+a_-)} + \frac{h_k\po\x(a+a_-)\pf^2}{\alpha (a+a_-)}+4(1+h_k)U''(a+a_-), \]
so that everything comes down to
\begin{eqnarray*}
 F_t''' + A F_t'' +2(1+h_k)\po G_t' + \lambda_k G_t\pf  & \leq& <H f_\perp, f_\perp> - \frac{k}{c_U}\|\x f_V\|^2.
\end{eqnarray*}
Under Assumtion \ref{a}, in one hand Lemma \ref{GtFtcoercif} gives $\|\x f_V\|^2 \geq c_U\| f_V \|^2$ and, in the other hand $H$ is bounded
; so there exists $\tau_k$ such that
\[H - \tau_k (a+a_-) + k\leq 0.\] 
Thus, for all $\tau \geq \tau_k$,
\begin{eqnarray*}
 <H f_\perp, f_\perp> & \leq & -\tau F_t'- k \| f_\perp\|^2.
\end{eqnarray*}
Finally we get
\begin{eqnarray*}
 F_t''' + A F_t'' + \tau F_t' + k F_t  + 2(1+h_k)\po G_t' + \lambda_k G_t\pf  \leq 0.
\end{eqnarray*}

\end{proof}

Here ends the proof that \eqref{equationORdre3} is satisfied for the telegraph model:
\begin{thm}\label{thmCTracinecommune}
Under assumption \ref{U}, there exist $\lambda,\eta>0$, and $t\mapsto \nu_t\geq \nu_*\in\mathbb R$ with $\mathcal Re(\eta - \sqrt{-\nu_*})>0$ such that
\[(\partial_t + \lambda)\Big[ (\partial_t + \eta)^2 + \nu_t\Big]F_t\leq 0.\] 
\end{thm}

\begin{proof}
We keep the notations used in the proof of Lemma \ref{lemCTinegalite}; our purpose is to find some parameters for wich
\begin{eqnarray*}
 \tilde Q_1(X) & = & 2(1+h_k)\po X + \lambda_k\pf
\end{eqnarray*}
and 
\begin{eqnarray*}
\tilde Q_3(X) & = & X^3 + AX^2  + \tau X + k 
\end{eqnarray*}
have a common root. Let $\beta \in(0, 1]$, $\alpha\in(0,1)$, $A \in \mathbb R $ be fixed, we let $k$ evolve in ${[0, 4 c_U (1 - \beta)]}$. The root of $ \tilde Q_1$,
\[-\lambda_k = -a_*\frac{2(1 - \beta) - \frac{k}{2c_U}}{(1+h_k)}\]
 is zero for $ k = 4 c_U (1 - \beta)$, else negative. We take $\tau \geq \tau_{4 c_U (1 - \beta)}$ large enough so that, for any $k \in {[0, 4 c_U (1 - \beta)]}$,  $\tilde Q_3$ has a unique non-positive real root, which is continuous with respect to $k$. This real root is zero for $k=0$ and negative otherwise, thus by continuity there exists a $k\geq 0$ such that $\tilde Q_3$ and $\tilde Q_1$ have a common root. We call $-\lambda$ this root and consider $u,v,w\in\mathbb R$ such that
\begin{eqnarray*}\label{equationlocale}
 \tilde Q_3(\partial_t)F_t + \tilde Q_1(\partial_t) G_t & = & (\partial_t + \lambda) \po F_t'' + u F_t' + v F_t + w G_t \pf\\
 & = & (\partial_t + \lambda) \po F_t'' + \po u + w\frac{c_U}{2a_*}\pf F_t' + \po v + w\po G_t - \frac{c_U}{2a_*}F_t'\pf \frac 1{F_t}\pf F_t \pf.
\end{eqnarray*}
Let $\eta =  \frac{2u+a_*^{-1}c_U}{4}$, $\nu_* = v + w c_U - \eta^2$ and
\begin{eqnarray*}
\nu_t& = & \po v + w\po G_t - \frac{c_U}{2a_*}F_t'\pf \frac 1{F_t}\pf - \eta^2.
\end{eqnarray*}
\eqref{definitiongtCT} exactly gives $\nu_t \geq \nu_*$. It remains to find some parameters for which $-\lambda$ and $\mathcal Re(-\eta \pm \sqrt{-\nu_*})$ are negative. These are the real parts of the roots of \[\tilde Q_3(X) +\frac {c_U}{2a_*} X \tilde Q_1(X) + c_U\tilde Q_1(X).\]
 Take $A >-c_U a_*^{-1}$, $\alpha\in(0,1)$, $\beta =1$, $\tau$ large enough and $k=0$, then  $h_k = \lambda_k = 0$ and zero is a common root. Thus $\tilde Q_1(X) = 2X$ and
\[\tilde Q_3(X) +\frac {c_U}{2a_*} X\tilde Q_1(X) + c_U\tilde Q_1(X) = X\po X^2 + \po A + \frac{c_U}{a_*}\pf X + \tau + 2 c_U\pf.\]
If $X^2 + \po A + c_U a_*^{-1}\pf X + \tau + 2 c_U$, polynomial with positive coefficients, has real roots, they are negative, and else $\mathcal Re(\eta \pm \sqrt{-\nu_*}) = \eta = \frac12\po A + c_U a_*^{-1}\pf >0$. Now if $\beta$ si slightly less than 1 this is still the case by continuity, but then $-\lambda$ is a real root of a polynomial with positive coefficient so $\lambda >0$.

\end{proof}

\section{Numerical studies}\label{numerique}
Although Theorem \ref{propconvergence} gives ``explicit'' bounds for the rate of convergence to equilibrium, it is not very easy to handle, as we get a set of polynomials $\Pi$ depending on several parameters in some set $\mathcal C$, such that the best rate obtained \emph{via} Theorem \ref{propconvergence} is
\[r = \underset{A,\beta,k,\tau \in \mathcal C}\max \min \left\{-\mathcal Re(\alpha),\ \alpha\ \text{root of }\Pi(X)\right\}\]
Using the notations of Lemma \ref{lemFPinegalite} and \ref{lemCTinegalite},
\[\Pi(X) = Q_3(X) + c_U\po\frac1{2 } X +1\pf Q_1(X)\]
for the kFP process, and
\[\Pi(X) =\tilde Q_3(X) + c_U\po\frac1{2 a_*} X +1\pf\tilde Q_1(X)\]
for the telegraph one.  $\mathcal C$ is the set of parameters for which $Q_1$ and $Q_3$ have a common root and the inequality is proven, for instance in the kFP model one need
\[0<\beta<1,\hspace{20pt} A>c_U,\hspace{20pt}k>0,\hspace{20pt}\tau \geq \underset{n\geq 1}{\max} \frac{P(-2n)+k}{2n}\]
($P$ defined in lemma \ref{lemFPinegalite}). Nevertheless we can numerically deal with this and compare the obtained results with the theorical rates when they are known, namely in the case of a quadratic potential for the kFP process (see \cite{Gadat2013}) and for the constant jump rate of the telegraph on the torus (see \cite{Volte-Face}). Obviously, such examples may just be considered as some benchmarks and are not really interesting processes for MCMC algorithm. As a consequence, once we will have seen the numerical rates can be of the right order of magnitude for the kFP model, we won't push this analysis deeper.

\bigskip

First of all, we adapt Lemma \ref{lemFPinegalite} in order to allow some changes in the parameters. The same computations lead to
\begin{lem}\label{lemFPinegalitevbU}
 Consider the generator
 \[L_{v,b,U} =  v\left(by\x - U'(x)\y\right) + \left(\y^2 -b y\y\right),\]
with invariant measure $e^{-U(x)}dx\otimes e^{-b\frac{y^2}{2}}$. Under Assumption \ref{U}, for any $A,k \in\mathbb R$ and $\beta >0$ there exists $\tau\in\mathbb R$ such that
 \[Q_3(\partial_t) F_t + Q_1(\partial_t) G_t \leq 0\]
with
\begin{eqnarray*}
Q_1(X) & =& v^2\beta\po X + 2b(\frac1\beta-1)-\frac k{2v^2c_U\beta}\pf\\ 
\\
Q_3(X) & = & X^3 + A X^2 + \tau X + k\\
\\
\tau & \geq & \underset{n\geq 1}{\max} \frac{P(-2bn)+k}{2bn}\\
\\
P(X) & = & X^3 +\po A+\frac{9b^2}\beta\pf X^2 +\left(6b^2\frac{3b+A}{\beta}+v^2\min\po (6b-4)U''\pf\right)X\\
& & + \frac1\beta(3b^2+Ab+v^2\beta||U''||_\infty)^2. 
\end{eqnarray*}
In the other hand,
\[G_t - \frac{c_U}{2 b} F_t' \geq c_U F_t\]
\end{lem}
In a MCMC algorithm, $U$ would be given while $b^{-1}$ the variance of the invariant speed and $v$ the ratio between the antisymmetric and symmetric parts of the dynamics should be chosen to get the fastest convergence to equilibrium (given the instantaneous randomness injected in the system).

The real exponential rate of convergence for $L_{v,b,U}$ with $U''$ constant is (see \cite{Gadat2013})
\[r_{theor} = 1 - \sqrt{\left(1-\frac{4v^2U''}{b}\right)_+}.\]

\begin{figure}
\centering
\includegraphics[scale=0.5]{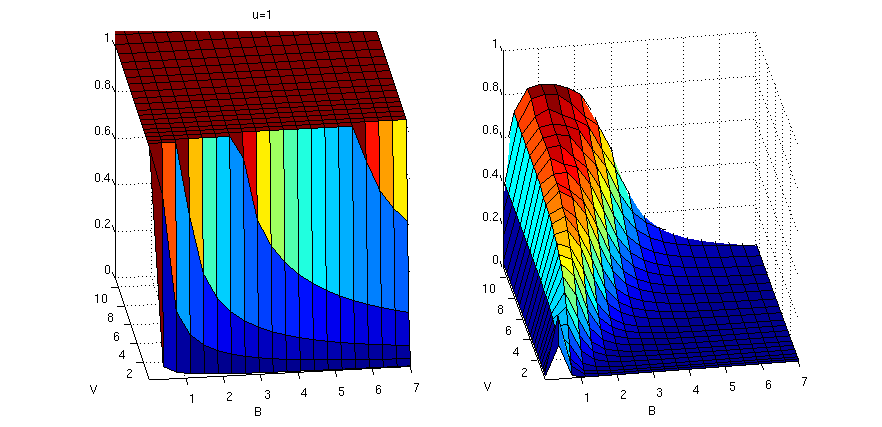}
 \caption{Left: theorical rate computed in \cite{Gadat2013}. Right: numerical rate given by Theorem \ref{thmFPracinecommune}. When $v$ is small (\emph{i.e.} the antisymetric part of $L$ is in a sense weak), the numerical rate is not very accurate and can miss the values for which the non reversible process is faster (asymptotically) than the reversible one. It becomes better with big values of $v$ and for some $b$ we get the right order of magnitude}\label{fig:imagetheor}
\end{figure}

Here are some numerical optimal rates given by Lemma \ref{lemFPinegalitevbU} (to be compared to the theorical one, in brackets if it is not 1) for $U'' = 1$ and different values of $v$ and $b$ (see also figure~\ref{fig:imagetheor}):

\bigskip

\begin{tabular*}{0.75\textwidth}{@{\extracolsep{\fill}}l|lllll}
$v\setminus b$ & 0.2 & 0.5 & 1  & 3 & 5 \\
\hline
0.5 & 0.11  & 0.21  & 0.02  & $03.10^{-4}$ (0.18) & $07.10^{-5}$ (0.10)\\
\hline
1 & 0.20  & 0.21  & 0.07  & $01.10^{-3}$  & $02.10^{-4}$ (0.55)\\
\hline
2 & 0.24  & 0.26  & 0.23  & $06.10^{-3}$  & $01.10^{-3}$ \\
\hline
10 & 0.26 & 0.32 & 0.56 & 0.23 & 0.03
\end{tabular*}

\section{Appendix}

\underline{Proof of Lemma \ref{lemDerivesdeFt}:}
\begin{proof}
From $\partial_t f_t = Lf_t$ comes
 \begin{eqnarray*}
 F_t' & = & 2 <Lf_t,f_t>\\
& = & <(2K)f_t,f_t>\\
F_t'' & = & 2 <L^2f_t,f_t> + 2 <Lf_t,Lf_t>
\end{eqnarray*}
From $(L+L^*)L = 2KL$ we compute
\begin{eqnarray*}
 F_t'' & =& <(2K)^2 f_t,f_t> + 4<KR f_t,f_t> - 4<KR^*f_t,f_t>\\
& = & <(2K)^2f_t,f_t> + 4 <[K,R]f_t,f_t>\\
 F_t''' & = & <(2K)^2Lf_t,f_t> + <L^*(2K)^2f_t,f_t> + 4 <[K,R]Lf_t,f_t>+ 4 <L^*[K,R]f_t,f_t>.
\end{eqnarray*}
Concerning the first two terms,
\begin{eqnarray*}
 (2K)^2L + L^*(2K)^2 & = & (2K)^3 + (2K)^2(R-R^*) - (R-R^*)(2K)^2\\
 & = & (2K)^3 + 4\po[K^2,R] - [K^2,R^*]\pf
\end{eqnarray*}
and $[K^2,R^*]^* = - [K^2,R]$. For the last two terms, $[K,R]K+K[K,R] = [K^2,R]$ thus
\begin{eqnarray*}
 4 <[K,R]Lf_t,f_t>+ 4 <L^*[K,R]f_t,f_t> & = & 4<[K^2,R] f_t,f_t> + 4 <\left[[K,R],R-R^*\right] f_t,f_t>.
\end{eqnarray*}
\end{proof}

\underline{Proof of Lemma \ref{lemFPderives}:}
\begin{proof}
First $\x$ and $\x^*$ commute with $\y$ and $\y^*$, and
\[\x\x^*  =  \x^*\x + U'',\hspace{20pt}\y\y^*  = \y^*\y + 1.\]
Now
\[\left[K,R\right] =  \y^*\y\x^*\y - \x^*\y\y^*\y =\po\y^*\y - \y\y^*\pf\x^*\y\\ = R \]
As well,
\begin{eqnarray*}
 \left[K^2,R\right] & =& -(\y^*\y)^2\x^*\y + \x^*\y(\y^*\y)^2\\
 & =& -\x^*(\y\y^* - 1)^2\y + \x (\y\y^*)^2\y\\
 & = & 2\x^*\y\y^* \y - \x^*\y\\
 & = & R(2K+1).
\end{eqnarray*}
Finally
\begin{eqnarray*}
 [R,R^*] & = & \x^*\y\y^*\x - \x\y^*\x^*\y\\
 &=& \x^*\x(\y^*\y+1) - (\x^*\x + U'') \y^*\y\\
 & = & -U''\y^*\y + \x^*\x\\
 & = & U'' K + \x^*\x.
\end{eqnarray*}
\end{proof}

\underline{Proof of Lemma \ref{lemDerivedeGt_kFP}:}
\begin{proof}
  \begin{eqnarray*}
 G_t' & =& 2<\x Kf_t,\x f_t> +2<\x (R-R^*) f_t,\x f_t>\\
& = & -2 <\x^*\x\y^*\y f_,f_t> + 2 <[\x^*\x,R]f_t,f_t>.
\end{eqnarray*}
For the first term
\[\x^*\x\y^*\y=\x^*\po \y\y^* - 1\pf\x=RR^* - \x^*\x,\]
and for the second one
\begin{eqnarray*}
\x^*\x R & = & -\x^*\x \x^*\y \\
 & = & -\x^*\po \x^*\x + U''\pf\y\\
 & =& R \x^*\x +R U'' .
\end{eqnarray*}
Finally,
\begin{eqnarray*}
 G_t' & = & <\left(-2 RR^* + 2\x^*\x + 2RU''\right)f_t,f_t> \\
 & = & <\left(-2 RR^* + 2RU''\right)f_t,f_t> + 2 <\x f_t,\x f_t>.
\end{eqnarray*}
\end{proof}

\underline{Proof of Lemma \ref{lemCTFtderives}:}
\begin{proof}
\[[K,R]= (a+a_-)\pi_\perp\pi_V y \x^* \pi_\perp - \pi_V y \x^* \pi_\perp (a+a_-)\pi_\perp= 0 + R(a+a_-)\]
The same computation holds for the second one :
 \begin{eqnarray*}
  [K^2,R] & =& -\po(a+a_-)\pi_\perp\pf^2\pi_V y \x^* \pi_\perp + \pi_V y \x^* \pi_\perp (a+a_-)^2\pi_\perp\\
  & = & 0 -R(a+a_-)^2
 \end{eqnarray*}
Since $(a+a_-)$ commutes with $\pi_V$, a term $\pi_V \pi_\perp$ appears in $R(a+a_-)R = R^2(a+a_-)=0$ we have
 \begin{eqnarray*}
  \left[\left[K,R\right],R-R^*\right] & = & \left[R(a+a_-),-R^*\right]\\
  & = & R^*R(a+a_-) - R(a+a_-)R^*.
 \end{eqnarray*}
Finally, as $\pi_\perp y \pi_V = y\pi_V$ and $y^2=1$ we get
\[RR^*= ( y \x^*\pi_\perp)\po \pi_\perp \x y \pi_V\pf= \x^*\x \pi_V\]
and similarly, as $\pi_V y  \pi_\perp = y\pi_\perp$,
\[R^*R=\x\x^*\pi_\perp= (\x^*\x + U'')\pi_\perp.\]
\end{proof}

\underline{Proof of Lemma \ref{CTGtderive}:}
\begin{proof}
 \begin{eqnarray*}
 G_t' & = & 2 <\x^*\x Lf_t,f_t>\\
 & =& 2<\x^*\x R f_t, f_t> - 2<\x^*\x R^* f_t, f_t> -2  <\x^*\x (a+a_-)\pi_\perp f_t,f_t>\\
&=& 2<[\x^*\x,R] f_t, f_t> - 2 <\x^*\x (a+a_-)\pi_\perp f_t,f_t>.
\end{eqnarray*}
For the first term,
\[[\x^*\x,R]=-\x^*\x\pi_V y \x^* \pi_\perp + \pi_V y \x^* \pi_\perp \x^*\x=R[\x,\x^*]=RU''\]
We conclude, according to Lemma \ref{lemCTFtderives}, with $\x^*\x\pi_\perp = R^*R - U''\pi_\perp$.
\end{proof}

\bibliographystyle{plain}
\bibliography{biblio}
\end{document}